\documentclass[12pt,reqno]{amsart}

\usepackage{fullpage}
\usepackage{amssymb,latexsym,mathrsfs,amsrefs}
\usepackage{url}
\usepackage{enumerate}[1.)]
\numberwithin{equation}{section}
\usepackage[colorlinks]{hyperref}

\renewcommand{\leq}{\leqslant}
\renewcommand{\geq}{\geqslant}

\numberwithin{equation}{section}

\newcommand{\Cc}{\mathbf{C}}

\newcommand{\Zz}{\mathbf{Z}}

\newcommand{\Rr}{\mathbf{R}}
\newcommand{\Gg}{\mathbf{G}}

\newcommand{\Qq}{\mathbf{Q}}
\newcommand{\Fp}{{\mathbf{F}_p}}

\newcommand{\Fpt}{{\mathbf{F}^\times_p}}

\newcommand{\Ff}{\mathbf{F}}

\newcommand{\bQl}{\bar{\Qq}_{\ell}}

\newcommand{\mods}[1]{\,(\mathrm{mod}\,{#1})}

\newcommand{\what}{\widehat}

\newcommand{\rmB}{\mathrm{B}}

\newcommand{\ra}{\rightarrow}
\newcommand{\lra}{\longrightarrow}

\DeclareMathOperator{\rank}{rank}

\DeclareMathOperator{\frob}{\mathrm{Fr}}

\DeclareMathOperator{\Tr}{tr}

\DeclareMathOperator{\swan}{Swan}

\DeclareMathOperator{\cond}{\mathbf{c}}

\newcommand{\eps}{\varepsilon}
\renewcommand{\rho}{\varrho}

\newcommand{\sheaf}[1]{\mathcal{{#1}}}

\DeclareMathSymbol{\gena}{\mathord}{letters}{"3C}
\DeclareMathSymbol{\genb}{\mathord}{letters}{"3E}

\theoremstyle{plain}
\newtheorem{theorem}{Theorem}[section]

\newtheorem{corollary}[theorem]{Corollary}

\newtheorem{proposition}[theorem]{Proposition}
\newtheorem*{proposition*}{Proposition}

\theoremstyle{remark}

\theoremstyle{definition}

\newtheorem{example}[theorem]{Example}
\newtheorem{remark}[theorem]{Remark}

\newcommand{\mcF}{\mathcal{F}}

\newcommand{\vphi}{\varphi}

\renewcommand{\geq}{\geqslant}
\renewcommand{\leq}{\leqslant}

\renewcommand{\le}{\leq}
\renewcommand{\ge}{\geq}
\begin{document}

\title{On short sums of trace functions}

\date{\today}

\author[\'E. Fouvry]{\'Etienne Fouvry}
\address{  Laboratoire de Math\'ematiques d'Orsay, Univ. Paris-Sud, Universit\' e Paris-Saclay, 
   91405 Orsay Cedex\\France}
\email{etienne.fouvry@math.u-psud.fr}

\author[E. Kowalski]{Emmanuel Kowalski}
\address{ETH Z\"urich -- D-MATH\\
  R\"amistrasse 101\\
  CH-8092 Z\"urich\\
  Switzerland} \email{kowalski@math.ethz.ch}

\author[Ph. Michel]{Philippe Michel}
\address{EPFL Mathgeom-TAN, Station 8, CH-1015 Lausanne, Switzerland }
\email{philippe.michel@epfl.ch}

\author[C.S. Raju]{Chandra Sekhar Raju}
\address{Department of Mathematics,
450 Serra Mall, Stanford,
California 94305, USA}
\email{chandras@stanford.edu}

\author[J. Rivat]{Jo\"el Rivat}
\address{Institut de Math{\' e}matiques de Marseille, Case 907\\
  Universit{\' e} d'Aix-Marseille\\ 163, avenue de Luminy\\ 13288
  Marseille Cedex 9\\ France}
\email{joel.rivat@univ-amu.fr}

\author[K. Soundararajan]{Kannan Soundararajan}
\address{Department of Mathematics,
450 Serra Mall, Stanford,
California 94305, USA}
\email{ksound@stanford.edu}

\thanks{Ph. M. was partially supported by the SNF (grant
  200021-137488); Ph.\ M.\ and E.\ K.\ were also partially supported
  by a DFG-SNF lead agency program grant (grant 200021L\_153647);
  \'E. F. thanks ETH Z\"urich, EPF Lausanne and the Institut
  Universitaire de France for financial support. Ph. M. thanks
  Stanford University, ETH Z\"urich and Caltech for providing
  excellent working conditions.  K. S. was partially supported by NSF
  grant DMS 1001068, and a Simons Investigator grant from the Simons
  Foundation. CS. R. was supported by B. C. and E. J. Eaves Stanford
  Graduate Fellowship. J. R. was supported by the ANR (grant 
  ANR-10-BLAN 0103).}



\begin{abstract}
  We consider sums of oscillating functions on intervals in cyclic
  groups of size close to the square root of the size of the group. We
  first prove non-trivial estimates for intervals of length slightly
  larger than this square root (bridging the ``Poly\'a-Vinogradov
  gap'' in some sense) for bounded functions with bounded Fourier
  transforms. We then prove that the existence of non-trivial
  estimates for ranges slightly below the square-root bound is stable
  under the discrete Fourier transform. We then give applications
  related to trace functions over finite fields.
\end{abstract}

\maketitle

\section{Introduction and statement of results} 

Consider a positive integer $m\geq 1$. Let
$\vphi: \Zz/m\Zz \rightarrow \Cc$ be a complex-valued function defined
modulo $m$, which we also view as a function $\Zz\lra \Cc$ by
composing with the reduction modulo $m$. Let $I\subset \Zz$ be an
interval of cardinality $|I|$ at most equal to $m$. One of the most
important general problems in analytic number theory is to estimate the
partial sum
$$
S(\vphi,I):=\sum_{n\in I}\vphi(n).
$$
\par
Of course, one has the obvious upper bound
\begin{equation}\label{trivialbound}
  |S(\vphi,I)|\leq \|\vphi\|_\infty |I|,
\end{equation}
where 
$$
\|\vphi\|_\infty=\max_{n\in\Zz/m\Zz}|\vphi(n)|
$$
but the goal is usually to improve significantly this bound when
$\vphi$ is an oscillating function.  Both the \emph{quality} of the
improvement (i.e., the saving compared with the trivial bound for
given $I$) and the \emph{range} of possibilities for $I$ that give
rise to non-trivial bounds are important. We will mostly focus on this
second aspect in this paper.


\subsection{The P{\' o}lya-Vinogradov range}

There exists a very general method to estimate incomplete sums, based
on Fourier theory on $\Zz/m\Zz$. This is called the
\emph{P{\' o}lya-Vinogradov} or \emph{completion} method.
\par
For any function $\vphi:\Zz\ra\Cc$ which is $m$-periodic, we define
its normalized Fourier transform $\what{\vphi}\,:\, \Zz\ra \Cc$ by
\begin{equation}\label{defFT}
  \what \vphi(h)=\frac{1}{\sqrt{m}}\sum_{x\mods m}\vphi(x)
  e\Bigl( \frac{hx}{m}\Bigr),
\end{equation}
for all $h\in\Zz$, where $e(t):=\exp(2\pi i{t})$.  We shall also find it convenient to use the notation $e_m(t)$ to denote $e(t/m) = e^{2\pi i t/m}$.  
\par
Given an interval $I\subset \Zz$ with cardinality at most $m$, let
$\tilde{I}$ be its image in $\Zz/m\Zz$, with characteristic function
denoted by ${\mathbf 1}_{\tilde{I}}$. The discrete Plancherel formula
gives the identity
$$
S(\vphi,I)=\sum_{h\mods m}\what \vphi(h) {\overline { \what {\mathbf
  1}_{\tilde{I}}(h)}},
$$
so that
\begin{equation}\label{PV}
  |S(\vphi,I)|\leq \|\what
  \vphi\|_\infty\sum_{h\mods m}\Bigl|\what {\mathbf 1}_{\tilde{I}}(h)\Bigr|\leq
  \|\what \vphi\|_\infty m^{1/2}\log(3m)
\end{equation}
since it is well known that
$$
\sum_{h\mods m}\Bigl|\what {\mathbf 1}_{\tilde{I}}(h)\Bigr|\leq
m^{1/2}\log(3m)
$$
for any $m\geq 1$ and any interval $I$. 
\par
Therefore if we assume that $\|\what{\vphi}\|_\infty\leq c$ for some
constant $c$, the bound \eqref{PV} will be non-trivial as long as
\begin{equation}\label{PVrange}
|I|\geq c m^{1/2}\log (3m)	
\end{equation}
which we call the \emph{P{\' o}lya-Vinogradov range}.
\par
The problem of estimating non-trivially such sums over shorter
intervals is crucial for many applications, for instance to study
averages or subconvexity estimates of $L$-functions \cites{Bu1,BFKMM}
(see also~\cite{k-sawin} for a recent very different situation where
this range determines the solution of a natural problem).   Here we study this problem, starting 
with a general result which gives a modest improvement over the P{\' o}lya-Vinogradov
range~\eqref{PVrange}.

\begin{theorem}\label{thmnolog} 
  For any interval $I$ in $\Zz$ with $\sqrt{m}<|I|\leq m$, we have
\begin{equation}\label{eq-slide}
  \Bigl|\sum_{n\in I}\varphi(n)\Bigr|\leq c\sqrt{m}\log
  \Bigl(\frac{4e^8|I|}{m^{1/2}}\Bigr)
\end{equation}
where $c=\max(\|\varphi\|_{\infty},\|\what{\varphi}\|_{\infty})$.
\end{theorem}

The estimate ~(\ref{eq-slide}) is non-trivial as soon as $I$ is of
length $\gg \sqrt m$, and we may view this result as ``bridging the
P{\' o}lya-Vinogradov gap''. As we will see in
Section~\ref{sec-slide}, the proof is very simple, but such results do
not seem to have been noticed before.  While we have given an explicit bound in 
Theorem \ref{thmnolog}, we have not made any attempt to optimize constants, and a more 
careful smoothing argument (for example, using the Beurling-Selberg trigonometric polynomials 
as in \cite{FroSo}) would provide better explicit constants.

Before continuing, we note that Theorem \ref{thmnolog} is essentially best
possible, since for $\varphi(n)=e({n^2}/m)$ the sum over
$1\leq n\leq m^{1/2}$ is $\gg m^{1/2}$.  Hence any improvement beyond
Theorem~\ref{thmnolog} requires some input on the function $\varphi$.

\subsection{Beyond the P{\' o}lya-Vinogradov range}\label{sec-fourier-t}

Our next result is concerned with the problem of going significantly
below the range $|I|\geq \sqrt{m}$ for suitable functions
$\vphi$. There are only few results of this type already known, the
most famous being the Burgess bound, when $\varphi(n)=\chi(n)$ is a
primitive Dirichlet character.
\par
We are currently unable to obtain results in great generality, but we
will obtain a number of new cases by proving a general principle that,
roughly speaking, states that if  the partial sums of a function $\varphi$ has substantial 
cancellation near the P{\' o}lya-Vinogradov range, then so does its discrete Fourier transform.  We now formulate this principle precisely, giving in the next section several applications.  
\par
Suppose throughout that $\vphi: \Zz/m\Zz \lra \Cc$ is a periodic function with 
$$ 
c= \max(\Vert \vphi \Vert_{\infty}, \Vert \what \vphi \Vert_{\infty} ). 
$$ 
For any $N \geq 1$, we define the sum $S(\varphi, N)$ by the formula
$$
S(\varphi, N) =\sum_{1\leq n \leq N } \varphi (n),
$$
and if $N\leq -1$, then we put
$$
S(\varphi, N) =\sum_{N\leq n\leq -1} \varphi (n).
$$
Next define, for any $1\le N \le m/2$, 
 \begin{equation} 
 \label{defDelta} 
 \Delta(\vphi,N) =\frac{1}{\sqrt{m}} +  \max_{m/2\ge t\ge 1} \Big\{
 \min\Big(\frac 1{ct}, \frac 1{cN}\Big) 
 \Big(|S(\vphi,t)| + |S(\vphi,-t)|\Big) \Big\}.  
 \end{equation} 
 From the definition, it is clear that $\Delta(\vphi,N)$ is a
 non-increasing function of $N$, and also that $N\Delta(\vphi,N)$ is a
 non-decreasing function of $N$.  Further, the definition immediately
 gives
 \begin{equation} 
 \label{defDelta2} 
 \max_{t\le N}\Big( |S(\vphi,t)| + |S(\vphi,-t)|\Big) \le cN \Delta(\vphi,N), 
 \end{equation} 
 and 
 \begin{equation} 
 \label{defDelta3} 
 \max_{m/2\ge t \ge N} \frac{1}{t} \Big( |S(\vphi,t)| + |S(\vphi,-t)|\Big) \le  c\Delta(\vphi,N).
\end{equation} 
With this notation, our main theorem transfers bounds for
$\Delta(\vphi,m/N)$ into bounds for $\Delta(\what\vphi,N)$.
 
\begin{theorem} \label{fouriertransfer} Let $m$, $\varphi$, $c$ and $\Delta$ be as above.  For $2\le N\le m/2$ we have 
\begin{equation}
 \label{eqn3.12} 
 |S(\what \vphi,N)| + |S(\what\vphi, -N)| \ll c \sqrt{N} m^{\frac 14} \Delta\Big(\vphi, \frac{m}{N}\Big)^{\frac 12},
 \end{equation} 
 and 
$$ 
\Delta(\what \vphi,N) \ll \frac{m^{\frac 14}}{\sqrt{N} } \Delta \Big( \vphi, \frac mN\Big)^{\frac 12}. 
$$ 
In particular 
\begin{equation} 
\label{eqn1.11} 
\Delta(\what \vphi, \sqrt{m}) \ll \Delta(\vphi,\sqrt{m})^{\frac 12}. 
\end{equation} 
\end{theorem}

If we apply the bound of \eqref{eqn1.11} twice, we see that $\Delta(\what \vphi,\sqrt{m}) 
\ll \Delta(\vphi,\sqrt{m})^{\frac 12} \ll \Delta(\what \vphi, \sqrt{m})^{\frac 14}$, so that some loss 
in precision has occurred.  One may wonder if this loss in precision could be removed, perhaps 
by defining some other quantity rather than $\Delta$.  





\subsection{Applications}

Our methods apply best to functions modulo $m$ that are pointwise
small and whose Fourier transform is also small, in a precise
quantitative sense.  In analytic number theory, there is a plentiful
supply of such functions which arise naturally in applications: they
are given by \emph{Frobenius trace functions} modulo $m$.
\par
These functions originate in algebraic geometry, and their analytic
properties have been investigated systematically in recent years by
Fouvry, Kowalski and Michel especially (see
\cites{FKM1,FKM2,FKM3,counting-sheaves,gowers-norms} for instance). We
will recall briefly the definition in Section~\ref{sec-add}, referring
to~\cite{pisa} for a longer survey.

Basic examples of trace functions lead to the following application of
Theorem~\ref{thmnolog}, where we denote as usual by $\bar{x}$ the
inverse of $x$ modulo $p$ for $x\in\Fpt$.



\begin{corollary}[Equidistribution over short intervals]
\label{cor-rational}
Let $\beta$ be any function defined on positive integers such that
$1\leq \beta(p)\ra +\infty$ as $p\ra +\infty$, and for all $p$ prime,
let $I_p$ be an interval in $\Fp$ of cardinality
$|I_p|\geq p^{1/2}\beta(p)$.
\par
\emph{(1)} Let $f_1$, $f_2\in\Zz[X]$ be monic polynomials such that
$f=f_1/f_2\in\Qq(X)$ is not a polynomial of degree $\leq 1$. Then for
$p$ prime, the set of fractional parts
$$
\Bigl\{\frac{f(n)}{p}\Bigr\},\quad\quad n\in I_p,
$$
becomes equidistributed in $[0,1]$ with respect to Lebesgue measure as
$p\ra +\infty$, where $f(n)=f_1(n)\overline{f_2(n)}$ is computed in
$\Fp$ and defined to be $0$ if $n\mods p$ is a pole of $f$.
\par
\emph{(2)} For $p$ prime and $n\in\Fp$ (resp. $n\in\Fpt$), define the
Birch (resp. Kloosterman) angles $\theta_{3,p}(n)$
\emph{(}resp. $\theta_{-1,p}(n)$\emph{)} in $[0,\pi]$ by the
relations
\begin{align*}
\rmB_3(n)=  \frac{1}{\sqrt{p}}\sum_{x\in\Fp} e_p(x^3+nx) 
  =2\cos\theta_{3,p}(n),\\
  \frac{1}{\sqrt{p}}\sum_{x\in\Fpt}e_p(\bar{x}+nx)
  =2\cos\theta_{-1,p}(n).
\end{align*}
Then the angles $\{\theta_{3,p}(n),\ n\in I_p\}$,
$\{\theta_{-1,p}(n),\ n\in I_p-\{0\}\}$ become equidistributed in
$[0,\pi]$ with respect to the Sato-Tate measure
$2\pi^{-1}\sin^2\theta d\theta$.
\end{corollary}

\begin{remark}
  We use the terminology ``Birch angle'' as analogous for Kloosterman
  angles. Historically, Birch~\cite[\S 3]{birch} mentioned the problem
  of the distribution of these angles as a problem similar to the
  Sato-Tate distribution of the number of points on elliptic curves
  over finite fields.  This Sato-Tate equidistribution was
  subsequently first proved by Livn\'e~\cite{livne}.
\end{remark}

See Section~\ref{sec-add} for the proofs, which are direct
applications of the Weyl criterion and the
estimate~(\ref{eq-slide}). These statements can be generalized
considerably to other summands, as will be clear from the proof in
Section~\ref{sec-add}; there are also variants for geometric
progressions instead of intervals, which are discussed in
Section~\ref{sec-mult}.

Below the P{\' o}lya-Vinogradov range, we obtain:

\begin{corollary}\label{cor-below-pv}
  Let $p$ be a prime number, let $P(X)\in\Fp[X]$ be a non-zero polynomial
  and let $\chi:\Fpt\ra\Cc^\times$ be a multiplicative character.
Assume that either $\chi$ is non-trivial or that $\deg P\geq 3$. Let
\begin{equation}\label{eq-function}
\varphi(x)=\chi(x)e_p(P(x))
\end{equation}
and
$$
\what\vphi(n)=\frac{1}{\sqrt{p}}\sum_{x\mods
  p}\chi(x)e_p(P(x)+{nx}).
$$
There exists $\delta>0$, depending only on $\deg P$ such that for any
interval $I\subset\Rr$ with $|I| \ge p^{\frac 12-\delta}$ we have
$$
\Bigl|\sum_{n\in I}\what{\varphi}(n)\Bigr|\ll |I|^{1-\delta}
$$
where the implied constant depends only on $\deg(P)$. 
\end{corollary}

The basic input here is the work of Weyl, Burgess, Enflo,
Heath-Brown, Chang and Heath-Brown--Pierce on short sums with summands
of the type $\chi(x)e(P(x)/p)$. 


Corollary \ref{cor-below-pv} has partial consequences to the distribution properties
of the cubic Birch sums in shorter intervals than is allowed in
Corollary~\ref{cor-rational}:

\begin{corollary}\label{cor-distrib-birch}
Let $p$ be a prime number and let $\rmB_3(n)$, the cubic Birch sum, be as in Corollary \ref{cor-rational}. 
There exists $\delta>0$, such that for any
interval $I\subset\Rr$ with $|I| \ge p^{\frac 12-\delta}$ we have
\begin{equation}\label{eq-mom1-birch}
\sum_{n\in I}\rmB_3(n)\ll |I|^{1-\delta}
\end{equation}
and
\begin{equation}\label{eq-mom2-birch}
\sum_{n\in I}|\rmB_3(n)|^2=|I|+O(|I|^{1-\delta})
\end{equation}
where the implied constants are absolute.  
Further, for such intervals $I$, and any $0\le t <\frac 12$ we 
have 
$$ 
\min \Big( \sum_{\substack { n\in I \\ \rmB_3(n)>t}} 1,
\sum_{\substack{n\in I\\ \rmB_3(n) <-t}} 1 \Big) \ge \Big(
\frac{1-2t}{4(2-t)}+o(1)\Big) |I|,
$$ 
and for any $0\le t<1$ we have 
$$ 
\sum_{\substack{n \in I \\ |\rmB_3(n)| > t }} 1 \ge \Big(
\frac{1-t^2}{4-t^2} + o(1) \Big) |I|.
$$ 
 \end{corollary}

 Using delicate work of Bourgain and Garaev on Kloosterman fractions
 \cite{BoGa}, \cite{BouPell}, we obtain corresponding, but weaker,
 results for short sums of Kloosterman sums:

\begin{corollary}\label{BouGarFou}
  Let $p$ be a prime number. For $k\geq 1$ an integer and $(a,p)=1$
  some parameter, let
$$
\varphi(x)=e_p(ax^{-k})
$$
for $x\in\Fpt$ and $\varphi(0)=0$. Let
$$
\what\vphi(n)=\frac{1}{\sqrt{p}}\sum_{x\mods
  p}e_p(ax^{-k}+{nx})
$$ 
be its Fourier transform.
\par
For $k=1$ and $k=2$, there exists $\delta>0$ such that for all $x\ge \sqrt{p} (\log p)^{-\delta}$ we have
$$
\Bigl|\sum_{1\leq n\leq x}\what\vphi(n)\Bigr| \ll x(\log x)^{-\delta},
$$
 where the implied constant
is absolute.
\end{corollary}

All these applications of Theorem~\ref{fouriertransfer} are found
in Section~\ref{sec-below-pv1}.


\subsection*{Acknowledgments.} We thank the referee for his or her
careful reading of the text, and especially for pointing out a slip in
the proof of Theorem~\ref{fourierpoly} in the first version of this
paper.

\section{Bridging the P{\' o}lya--Vinogradov gap}\label{sec-slide}

\subsection{The basic inequality: Proof of Theorem~\ref{thmnolog}}



We may assume that $m\geq 64$ since otherwise the trivial bound $cm$
is better than the claim.  In that case we have
$\lceil \sqrt{m}\rceil\leq m/4$.
\par
For any $r\geq 0$, we denote by $I_r=r+I$ the interval $I$ shifted by
$r$, and by $\what{I}_r$ the Fourier transform of its characteristic
function:
$$
\what{I}_r(t)=\frac{1}{\sqrt{m}}
\sum_{x\in I_r}e_m(tx).
$$
\par
By the discrete Plancherel formula, we have
$$
S(\varphi,I_r)=\sum_{t\in\Zz/m\Zz} \what\vphi(t)\overline{\what{I}_r(t)}=
\sum_{t\in\Zz/m\Zz}\what\vphi(t)\overline{\what{I}(t)}e_m (-rt)
$$
for any $r$, where $\what{I}=\what{I}_0$. Moreover, we have
$$
|S(\varphi,I_r)-S(\varphi,I)|\leq 2cr
$$
since $ |\varphi(x) | \leq c$ for all $x$. 
\par
Let $R=\lceil \sqrt{m}\rceil$. Since $\sqrt{m}<|I|$, we see that $R$
is an integer with $\sqrt{m} \le R \le |I|$.  Thus,
$$
|I|\geq R\geq \sqrt{m}\geq m/|I|\geq 1,\quad\quad m\geq 4R
$$ 
(the last inequality because $m\geq 64$, as assumed at the beginning
of the proof).  Summing our identity for $S(\vphi,I_r)$ for
$1\leq r\leq R$, we obtain
\begin{equation}\label{eq-1}
RS(\varphi,I)= \sum_{-m/2<t\leq m/2}\what\vphi(t) \overline{\what{I}(t)}
\sum_{1\leq r\leq R}e_m(-{rt}) +E,
\end{equation}
where $|E|\leq 2cR^2$.
\par
Now, the Fourier transform $\what{I}$ satisfies
$$
|\what{I}(t)|\leq \frac{1}{\sqrt{m}}\min\Bigl(|I|,\frac{m}{2|t|}\Bigr)
$$
for $-m/2\leq t\leq m/2$, and  similarly, we have
$$
\Bigl|\sum_{1\leq r\leq R}e_m (-rt) \Bigr| \leq
\min\Bigl(R,\frac{m}{2|t|}\Bigr).
$$
Using these bounds in \eqref{eq-1}, together with $R\leq |I|$ and
$|\what\vphi(t)|\leq c$, we get
\begin{equation*}
  R |S(\varphi,I)| 
  \leq c\Bigl\{\sum_{|t|\leq m/(2|I|)}
    R\frac{|I|}{m^{1/2}} +
    \sum_{m/(2|I|)< |t|\leq m/(2R)} R\frac{m^{1/2}}{2|t|}
    +\sum_{m/(2R) <|t|\leq
    m/2} \frac{m^{3/2}}{4t^2} \Bigr\} + 2cR^2.
\end{equation*}
\par
The first sum above is at most
$$
\frac{R|I|}{\sqrt{m}}\Bigl(\frac{m}{|I|}+1\Bigr)\leq 2R\sqrt{m}.
$$
Since $m\ge 4R$, the third term is at most
$$
\frac{m^{3/2}}{2}\sum_{t>m/(2R)}\frac{1}{t^2}\leq
\frac{m^{3/2}}{2}\frac{1}{m/(2R)-1} =\frac{m^{3/2}R}{m-2R} \le 2 R\sqrt{m}. 
$$
We claim that the middle term is
\begin{equation} 
\label{eqn2.2} 
R\sqrt{m}\sum_{m/(2|I|)< t \leq m/(2R)} \frac{1}{t} \le R \sqrt{m} \log \Big( \frac{4|I|}{\sqrt{m}}\Big),
\end{equation} 
from which it follows that
$$ 
|S(\vphi,I)| \le 4c\sqrt{m} + 2cR + c\sqrt{m} \log \Big( \frac{4|I|}{\sqrt{m}}\Big) 
\le c \sqrt{m}\log \Big( \frac{4|I|}{\sqrt{m}} \Big) + 8c\sqrt{m},
$$ 
as desired.  

To verify the claim \eqref{eqn2.2}, note that if $|I| <m/4$ then the 
quantity in question is 
$$ 
\le R\sqrt{m} \log \Big( \frac{m/(2R)}{m/(2|I|)-1} \Big) \le R\sqrt{m} \log \Big( \frac{m/(2R)}{m/(4|I|)}\Big), 
$$ 
which verifies \eqref{eqn2.2} in this range.  
If $m/4\leq |I|< m/2$, then we may use the bound 
$$
R\sqrt{m} \sum_{1< t\le m/(2R)} \frac 1t \le R\sqrt{m}\log\Bigl(\frac{m}{2R}\Bigr)\leq
R\sqrt{m}\log\Bigl(\frac{4|I|}{\sqrt{m}}\Bigr), 
$$
which again verifies
\eqref{eqn2.2}. 
Finally, if $|I|\geq m/2$, then
$$
R\sqrt{m} \sum_{t\le m/(2R)} \frac 1t \leq R\sqrt{m}\Bigl(1+\log \Bigl(\frac{m}{2R}\Bigr)\Bigr) \leq
R\sqrt{m}\log\Bigl(\frac{4|I|}{\sqrt{m}}\Bigr), 
$$
which completes our verification of \eqref{eqn2.2}.  


\subsection{Applications to trace functions, I: the additive case}\label{sec-add}

We now recall the definition and give some basic examples of trace
functions before proving Corollary~\ref{cor-rational}.  As is usual,
we will restrict our attention to prime moduli; the extension of the
results to squarefree moduli at least is a matter of applying the
Chinese Remainder Theorem.

Thus let $p$ be a prime number. Given a prime $\ell\not=p$, we fix an
isomorphism $\iota\,:\, \bar{\Qq}_{\ell}\simeq \Cc$, and we use it
implicitly to identity any $\ell$-adic number with a complex number. A
\emph{Fourier sheaf} modulo $p$ is defined to be a middle-extension
$\bQl$-adic sheaf $\sheaf{F}$ on $\mathbf{A}^1_{\Fp}$, that is
pointwise pure of weight $0$ and of Fourier type, i.e., none of its
geometric Jordan-H\"older components is isomorphic to an
Artin-Schreier sheaf $\sheaf{L}_{\psi}$ for some additive character
$\psi$.

\begin{remark} 
  Note that in contrast with the definition of Katz~\cite{katz-gskm},
  we impose the weight $0$ condition instead of stating it separately.
\end{remark}

The \emph{(Frobenius) trace function} of $\mcF$ is the function
$\Zz/p\Zz\lra \Cc$ defined by
\begin{equation}\label{twistedphase}
 \vphi=\iota(\Tr(\frob_{x,\Fp}|\mcF))	
\end{equation}
for any $x\in\Fp$. It is a deep property, due to Deligne, that the
Fourier transform of the trace function of $\mcF$ is also a trace
function, namely that of the sheaf-theoretic (normalized) Fourier transform
of $\mcF$.

The complexity of the trace function is controlled by the (analytic)
conductor of the sheaf $\mcF$, which is defined as
$$
\cond(\sheaf{F})=\rank(\sheaf{F})+n(\sheaf{F})+\sum_{x\in
  S(\sheaf{F})}\swan_x(\sheaf{F}),
$$
where $\rank(\sheaf{F})$ is the rank of $\sheaf{F}$,
$S(\sheaf{F})\subset \mathbf{P}^1(\bar{\Ff}_p)$ is the set of
singularities of $\sheaf{F}$, the integer $n(\sheaf{F})\geq 0$ is the
cardinality of $S(\sheaf{F})$ and $\swan_x$ denotes the Swan conductor at such a
singularity. The conductor is a non-negative integer, and from its properties (see~\cite[Prop. 8.2]{FKM1} for \eqref{condbound}) 
 we have the following inequalities
\begin{equation}\label{rankcondbound}
\|\vphi\|_\infty\leq\rank(\mcF)\leq\cond(\mcF)	
\end{equation}
and
\begin{equation}\label{condbound}
	\cond(\what\mcF)\leq 10\cond(\mcF)^2.
\end{equation}
Thus we obtain
\begin{equation}\label{tracefctbound}
  \max(\|\vphi\|_\infty,\|\what\vphi\|_\infty)\leq 10\cond(\mcF)^2.
\end{equation} 
\par
This means that, for instance, Theorem~\ref{thmnolog} or
Theorem~\ref{fouriertransfer} can be applied efficiently to a
sequence of trace functions modulo primes $p\ra +\infty$, provided the
conductor of the underlying sheaves is bounded independently of $p$.

\begin{example}
  Let $f_1$, $f_2\not=0$, $g_1$, $g_2\not=0$ be monic polynomials with
  integer coefficients such that $f_1$ is coprime to $f_2$ and $g_1$
  is coprime to $g_2$. Let $p$ be a prime number and $\chi$ a
  multiplicative character modulo $p$. If $\chi$ is trivial, we adopt the convention 
  that $f_1=f_2=1$.  If $\chi$ is non-trivial, we assume that none of the zeros or poles of $f_1/f_2$ has order divisible by  the order of $\chi$.  Define the rational functions $f=f_1/f_2$ and
  $g=g_1/g_2$, and let
$$
\varphi(n)=\chi\bigl(f(n)\bigr)e_p\bigl(g(n)\bigr)
$$
for $n\in\Zz$ such that 
$$
f_1(n)f_2(n)g_2(n)\not=0\mods{p},
$$
where $f(n)=f_1(n)\overline{f_2(n)}$ and
$g(n)=g_1(n)\overline{g_2(n)}$ are computed in $\Fp$, and let
$$
\varphi(n)=0
$$
if $f_1(n)f_2(n)g_2(n)=0\mods{p}$.
\par
For all primes $p$ large enough, the poles of $g$ are of order $< p$.  
For any such prime, the function $\varphi$ is the trace function of a
middle-extension sheaf $\sheaf{F}$ with
$$
\cond(\sheaf{F})\ll \deg(f_1)+\deg(f_2)+\deg(g_1)+\deg(g_2),
$$
where the implied constant is absolute. If $f$ is not constant modulo
$p$ or if $g$ is not a polynomial of degree at most $1$, this sheaf is a
Fourier sheaf.
\end{example}

\begin{proof}[Proof of Corollary~\ref{cor-rational}]
  (1) We can certainly assume that $\beta(p)<p^{1/2}$ for all $p$.  By
  the Weyl criterion, we must show that, for any fixed integer
  $h\not=0$, and for the interval $I_p$, the sums
$$
\frac{1}{|I_p|}\sum_{n\in I_p}e_p\bigl(hf(n)\bigr)
$$
tend to $0$ as $p\ra +\infty$. For a given $p$, and a suitable
$\ell$-adic non-trivial additive character $\psi$ of $\Fp$, there
exists a rank $1$ sheaf 
$\sheaf{F}=\sheaf{L}_{\psi(hf(X))}$
with trace function given by
$$
\vphi(x)=e_p(hf(x))
$$
for all $x\in\Fp$. This is a middle-extension sheaf modulo $p$, which
is 
pointwise pure of weight $0$. For $p$ large enough so that $hf(X)$ is
not a polynomial of degree $\leq 1$, this sheaf is a Fourier
sheaf. Its conductor satisfies
$$
\cond(\sheaf{F})\leq 1+(1+\deg(f_2))+\sum_{x\text{ pole of } f_2}
\mathrm{ord}_{x}(f_2)+\deg(f_1)\ll 1
$$
for all $p$ large enough (the first $1$ is the rank, the singularities
are at most at poles of $f_2$ and at $\infty$, the Swan conductor at a
pole of $f_2$ is at most the order of the pole, and at infinity it is
at most the order of the pole of $f$ at infinity, which is at most the
degree of $f_1$). Hence, there exists $c\geq 1$ such that the trace
function $\vphi$ satisfies
$$
\max(\| \vphi \|_{\infty},\|\what \vphi\|_{\infty})\leq
c
$$
for all large $p$. By Theorem \ref{thmnolog}, we get
$$
\frac{1}{|I_p|}\sum_{n\in I_p}e_p(hf(n))=\frac{1}{|I_p|}
S(\vphi,I_p) \ll \frac{\sqrt{p}}{|I_p|}\log
\Bigl(\frac{|I_p|}{p^{1/2}}\Bigr) \ll 
\frac{\log\beta(p)}{\beta(p)}\ra 0
$$
by assumption.
\par
(2) Let $\theta_p=\theta_{-1,p}$ or $\theta_{3,p}$, depending on
whether one considers Kloosterman sums or Birch sums. Using the Weyl
criterion, and keeping some notation from (1), it is enough to show
that for any fixed $d\geq 1$, we have
$$
\lim_{p\rightarrow +\infty} \frac{1}{|I_p|}\sum_{n\in
  I_p}U_d(2\cos\theta_p(n))=0
$$
where $U_d\in\Zz[X]$ is the Chebyshev polynomial defined by 
$$
U_d(2\cos\theta)= \sin((d+1)\theta) / \sin\theta .
$$
\par
By the theory of the Fourier transform of sheaves (see the exposition
in~\cite[Ch. 8]{katz-gskm} and the survey in~\cite[\S 10.3]{FKM1}),
the function
$$
\varphi(x)=U_d(2\cos\theta_p(x))
$$
is the trace function of an $\ell$-adic irreducible middle-extension
Fourier sheaf (the symmetric $d$-th power of the rank $2$ Kloosterman
sheaf or of the Fourier transform of the sheaf
$\mathcal{L}_{\psi(x^3)}$, which is also of rank $2$, both of which
are irreducible); this sheaf has rank $d+1\geq 2$ on the affine line
over $\Fp$, and its conductor is bounded by a constant depending only
on $d$, and not on $p$. It is therefore a Fourier sheaf with trace
function satisfying
$$
\max(\|\varphi\|_{\infty},\|\what{\varphi}\|_{\infty})\leq c
$$
for some $c$ depending only on $d$, and hence the desired limit holds
again by a direct application of Theorem \ref{thmnolog}. 
\end{proof}

Another interesting and somewhat similar application is the following:

\begin{proposition}[Polynomial residues]\label{pr-pol-residues}
  Let $\beta$ be a function defined on integers such that $1\leq
  \beta(m)\ra +\infty$ as $m\ra +\infty$. Let $f\in \Zz[X]$ be a
  non-constant monic polynomial. For all primes $p$ large enough,
  depending on $f$ and $\beta$, and for any interval $I_p$ modulo $p$
  of size $|I_p|\geq p^{1/2}\beta(p)$, there exists $x\in I_p$ such
  that $x=f(y)$ for some $y\in \Fp$. In fact, denoting by $P$ the set
  $f(\Fp)$ of values of $f$, the number of such $x$ is $\sim \delta_f
  |I_p|$ as $p\ra +\infty$, where $\delta_f=|P|/p$.
\end{proposition}

Here again, the interest of the result is when $\beta(m)$ is smaller
than $\log m$. However, it seems likely that this distribution
property should also be true for much shorter intervals (as in the
well-known conjecture for quadratic (non)-residues).

\begin{proof}
  Let $\varphi$ be the characteristic function of the set $P$ of
  values $f(y)$ for $y\in\Fp$. We must show that, for $p$ large
  enough, we have
$$
\sum_{x\in I_p}\varphi(x)\sim \delta_f |I_p|
$$
(which in particular implies that the left-hand side is $>0$ for $p$
large enough.)
\par
By~\cite[Prop. 6.7]{FKM2}, if $p$ is larger than $\deg(f)$, there
exists a decomposition
$$
\varphi(x)=\sum_i c_i \varphi_i(x)
$$
where the number of terms in the sum and the $c_i$ are bounded in
terms of $\deg(f)$ only, and where $\varphi_i$ is the trace function
of a \emph{tame} $\ell$-adic middle-extension sheaf $\sheaf{F}_i$ with
conductor bounded in terms of $\deg(f)$ only. Moreover, $\sheaf{F}_1$
is the trivial sheaf with trace function equal to $1$, all others are
geometrically non-trivial and geometrically isotypic, and
$$
c_1=\delta_f+O(p^{-1/2}),
$$
where $\delta_f=|P|/p$ and the implied constant depends only on
$\deg(f)$. In particular, $\sheaf{F}_i$, being tame and geometrically
isotypic and non-trivial, is a Fourier sheaf for $i\not=1$. We also
note that $\delta_f\gg 1$ for primes $p>\deg(f)$.
\par
This decomposition implies
$$
\sum_{x\in I_p}\varphi(x)=c_1|I_p| +\sum_{i\not=1} c_i
S(\vphi_i,I_p)= \delta_f|I_p| +\sum_{i\not=1} c_i
S(\vphi_i,I_p)+O(p^{-1/2}|I_p|).
$$
\par
Since the $\sheaf{F}_i$, for $i\not=1$, are Fourier sheaves, we get by Theorem \ref{thmnolog}
$$
S(\vphi_i,I_p)\ll \sqrt{p}\log
\Bigl(\frac{|I_p|}{p^{1/2}}\Bigr) \ll |I_p|
\frac{\log\beta(p)}{\beta(p)},
$$
for each $i\not=1$, where the implied constant depends only on
$\deg(f)$. Hence we obtain
$$
\sum_{x\in I_p}\varphi(x)\sim \delta_f|I_p|
$$
uniformly for $p>\deg(f)$, since $\beta(p)\ra+\infty$, which gives the
result.
\end{proof}

\subsection{Applications to trace functions, II: the multiplicative
  case}\label{sec-mult}

We consider now a different application of the basic inequality: for a
prime $p$, we look at the values of trace functions modulo $p$ on the
multiplicative group $A=\Fpt\simeq \Zz/(p-1)\Zz$.  Fixing a generator
$g$ of $A$, this means that we are now looking at sums over
\emph{geometric progressions} $xg^n$ for $n$ in some interval $I$ in
$\Zz/m\Zz=\Zz/(p-1)\Zz$.  Such sums have also been considered by Korobov,
for instance (see, e.g.,~\cite[Ch. 1, \S 7]{korobov}).
\par
We will use the notation and terminology of the previous section, but
to avoid confusion we write $\tau_{\sheaf{F}}$ for the restriction of
the trace function of a sheaf $\sheaf{F}$ to $\Fpt$. The discrete
Fourier transform becomes the discrete Mellin transform
$$
\widehat{\tau}_{\sheaf{F}}(\chi)=\frac{1}{\sqrt{p-1}} \sum_{x\in
  \Fpt}{\tau}_{\sheaf{F}}(x)\chi(x)
$$
defined for $\chi$ in the group of multiplicative characters of
$\Fpt$. (More precisely, this Mellin transformed can be identified
with the discrete Fourier transform on $\Fpt\simeq \Zz/(p-1)\Zz$; as
we are interested in bounds for the maximum of the Fourier transform,
we may as well use the multiplicative characters as arguments).

\par
The analogue of Fourier sheaves in this case are the sheaves with
``property $\mathcal{P}$'' of Katz's work on the discrete Mellin
transform~\cite[Chapter 1]{katz-mellin}.

\begin{proposition}\label{pr-no-correlation-mult}
  Let $p$ be a prime number, and let $\sheaf{F}$ be an $\ell$-adic
  middle extension sheaf modulo $p$ with conductor $c$, pointwise pure
  of weight $0$.  If no geometric Jordan-H\"older component of
  $\sheaf{F}$ is isomorphic to a Kummer sheaf $\sheaf{L}_{\chi}$
  associated to a multiplicative character $\chi$, then the Mellin
  transform of the trace function $\tau_{\sheaf{F}}$ is bounded by
  $2\sqrt{2}c^2$, i.e., for any character $\chi$ of $\Fpt$, we have
$$
\Bigl|\frac{1}{\sqrt{p-1}} \sum_{x\in
  \Fpt}{\tau}_{\sheaf{F}}(x)\chi(x) \Bigr|\leq 2\sqrt{2}c^2.
$$
\end{proposition}

\begin{proof}
This is again a form of the Riemann Hypothesis of Deligne.
By the Grothendieck-Lefschetz trace formula, and the assumption on
$\sheaf{F}$ which ensures the vanishing of $H^2_c$, we have
$$
\sum_{x\in \Fpt}{\tau}_{\sheaf{F}}(x)\chi(x)
=-\mathrm{Tr}(\mathrm{Frob}_{\Fp}\mid
H^1_c(\Gg_m\times\bar{\Ff}_p,\sheaf{F}\otimes\sheaf{L}_{\chi}) ).
$$
\par
Then, since the sheaf involved is pointwise mixed of weight $0$ on
$\Gg_m$, Deligne's Theorem implies that each eigenvalue of the
Frobenius acting on
$H^1_c(\Gg_m\times\bar{\Ff}_p,\sheaf{F}\otimes\sheaf{L}_{\chi}) $ has
modulus at most $\sqrt{p}$. Thus
$$
\Bigl|\sum_{x\in \Fpt}{\tau}_{\sheaf{F}}(x)\chi(x)\Bigr| \leq
(\dim H^1_c(\Gg_m\times\bar{\Ff}_p,\sheaf{F}\otimes\sheaf{L}_{\chi})
)) \sqrt{p}.
$$
\par
Let $U\subset \Gg_m$ be the maximal dense open subset where
$\sheaf{F}$ is lisse. By the Euler-Poincar\'e characteristic formula,
we have
\begin{multline*}
  \dim H^1_c(\Gg_m\times\bar{\Ff}_p,\sheaf{F}\otimes\sheaf{L}_{\chi})
  =-\chi_c(\Gg_m\times\bar{\Ff}_p,\sheaf{F}\otimes\sheaf{L}_{\chi})
\\  =\swan_0(\sheaf{F}\otimes\sheaf{L}_{\chi})
  +\swan_{\infty}(\sheaf{F}\otimes\sheaf{L}_{\chi})+ \sum_{x\in
    (\Gg_m-U)}(\mathrm{drop}_x(\sheaf{F}\otimes\sheaf{L}_{\chi})
\\  +\swan_x(\sheaf{F}\otimes\sheaf{L}_{\chi}))
\end{multline*}
(see, e.g.,~\cite[p. 67]{katz-mellin}). Since $\sheaf{L}_{\chi}$ is
tame of rank $1$, we have
$\swan_x(\sheaf{F}\otimes\sheaf{L}_{\chi})=\swan_x(\sheaf{F})$ for all
$x$, and therefore
$$
\swan_0(\sheaf{F}\otimes\sheaf{L}_{\chi})
  +\swan_{\infty}(\sheaf{F}\otimes\sheaf{L}_{\chi})+ \sum_{x\in
    (\Gg_m-U)}\swan_x(\sheaf{F}\otimes\sheaf{L}_{\chi})
\leq \cond(\sheaf{F})=c.
$$
\par
Furthermore, we have
$\mathrm{drop}_x(\sheaf{F}\otimes\sheaf{L}_{\chi})\leq
\rank(\sheaf{F})\leq c$ for all $x$, and at most $c$ points occur
where it is non-zero. Thus we derive
$$
\dim H^1_c(\Gg_m\times\bar{\Ff}_p,\sheaf{F}\otimes\sheaf{L}_{\chi})
\leq c+c^2\leq 2c^2.
$$
\par
Finally, since $\sqrt{p}\leq \sqrt{2}\sqrt{p-1}$ for all primes $p$,
we deduce the result.
\end{proof}

\begin{corollary} Let $c\geq 1$, let $p$ be a prime number, and let
  $\sheaf{F}$ be an $\ell$-adic middle extension sheaf modulo $p$,
  pointwise pure of weight $0$, with no Kummer sheaf as a geometric
  Jordan-H\"older component.  Assume that the conductor of $\sheaf{F}$
  is $\leq c$.

  Let $g\in\Fpt$ be a generator of $\Fpt$ and let $x\in \Fpt$ be
  given. For any interval $I$ in $\Zz/(p-1)\Zz$, we have
$$
\Bigl|\sum_{n\in I}\tau_{\sheaf{F}}(xg^n)\Bigr|\leq
2\sqrt{2}c^2{\sqrt{p-1}}\log\Bigl(4e^8\frac
{|I|}{\sqrt{p-1}}\Bigr).
$$
\end{corollary}

This follows immediately by combining
Proposition~\ref{pr-no-correlation-mult} and Theorem~\ref{thmnolog}
with $m=p-1$.
\par
From this result, we can deduce equidistribution statements exactly
similar to Corollary~\ref{cor-rational}, with geometric progressions
replacing intervals, since the sheaves used in the proof are in fact
geometrically irreducible and are not Kummer sheaves.
\par
Proposition~\ref{pr-pol-residues} also extends with some restriction
on the polynomial:

\begin{proposition}
  Let $\beta$ be a function defined on integers such that $1\leq
  \beta(n)\ra +\infty$ as $n\ra +\infty$. Let $f\in \Zz[X]$ be a
  non-constant \emph{squarefree} monic polynomial. For all primes $p$
  large enough, depending on $f$ and $\beta$, for any primitive root
  $g$ modulo $p$, and for any interval $I_p$ in $\Zz/(p-1)\Zz$ of size
  $|I_p|\geq (p-1)^{1/2}\beta(p)$, there exists $n\in I_p$ such that
  $g^n=f(y)$ for some $y\in \Fpt$.
\end{proposition}

The only change in the proof of the previous case is that we must
ensure that the sheaves $\sheaf{F}_i$ for $i\not=1$ appearing there
have no Kummer sheaf as Jordan-H\"older component, for $p$ large
enough.  This is indeed the case because the assumption that $f$ is
squarefree implies first that $f$ is squarefree modulo $p$ for $p$
large enough, and from this follows for each such prime that each
$\sheaf{F}_i$ is lisse at $0$ (see~\cite[Prop. 6.7]{FKM2}), which is
not the case of $\sheaf{L}_{\chi}$.

\section{The Fourier transfer principle}
\label{sec-transfer}

\subsection{The basic principle}

In this section we provide a quantitative version of the transfer
principle discussed in Section~\ref{sec-fourier-t}. The idea is to
estimate a sum
$$
\sum_{t\in I}\what\vphi(t)
$$
by beginning with an application of the completion method in a smooth
form, followed however by a summation by parts that allows us to
exploit bounds for short sums of the original function $\varphi$.
\par

\begin{proposition} \label{prop3.1}    Let $m\geq 2$ be an integer and let
$$
\varphi\,:\, \Zz/m\Zz\lra \Cc
$$
be an arbitrary function. Let $c\geq 1$ be such that
$$
\max(\|\varphi\|_{\infty},\|\what\vphi\|_{\infty})\leq c.
$$
For any $N$ with $|N|\le m/2$ we have 
$$ 
|S({\hat \vphi},N)| \ll cmU + c\frac{N}{\sqrt{m}} + \frac{cN}{m^3 U^3}+ \frac{N}{\sqrt{m}} \int_1^{m/2} \Big( |S(\vphi,t)| +|S(\vphi,-t)| \Big) \min \Big( \frac{N}{m}, \frac 1t, \frac{1}{t^4U^3} \Big) dt , 
$$ 
where the implied constant is absolute, and for any $U \in [1/m, N/(2m)]$. 
\end{proposition} 
\begin{proof}  We will view $\varphi$ as a function on $\Zz$ which is periodic
  modulo $m$.  We consider the case of positive $N$ with $1\le N\le m/2$, the 
  negative case being entirely similar. 
  
  \par
   Let $U\in [1/m,N/(2m)]\subset(0,1/4]$ be some parameter. We fix a smooth function
$\Psi\,:\, [0,1]\lra [0,1]$ with compact support contained in
$[U,N/m+U]\subset [0,1]$, such that
\begin{itemize} 
\item 
For $x\in [2U, N/m]$ we have $\Psi(x)=1$.  

\item The function $\Psi$ is increasing on the interval $[U,2U]$ and
  decreasing on the interval $[N/m,N/m+U]$, 
\item For any integer $l\geq 0$, we have
 $$
\Psi^{(l)}(\alpha)\ll_l U^{-l}.
$$
\end{itemize}

\par
We extend $\Psi$ to a $1$-periodic function on $\Rr$ and consider its
Fourier expansion
\begin{equation}\label{fourier}
  \Psi(\alpha)=\sum_{h\in\Zz}\what\Psi(h)e(-h\alpha),
\end{equation}
with
\begin{equation}\label{fouriercoeff}
  \what\Psi(h)=\int_{[0,1]}\Psi(\alpha)e(h\alpha)d\alpha.
\end{equation}
\par
The properties of the derivatives of $\Psi$ immediately imply the
following bounds for its Fourier coefficients
\begin{equation}\label{eq-psi0}
\what\Psi(0)=\frac{N-1}m+O(U)=O\Bigl(\frac Nm\Bigr),
\end{equation}
and
\begin{equation}\label{fouriercoeffbound}
\what\Psi(h)\ll_A \min\Bigl(\frac{N}{m},\frac{1}{|h|}({U|h|})^{-A}\Bigr)
\end{equation}
for any $h\not= 0$ and any $A\geq 0$.  Indeed, this follows from the
definition of the Fourier coefficients using repeated integrations by
parts and the fact that $\Psi$ is supported in an interval of length
$\ll N/m$, whereas the derivatives of $\Psi$ are supported in the
union of two intervals of length $U$.
\par
Now, we observe furthermore that the expression \eqref{fouriercoeff}
defines a \emph{smooth} function of $h$ on the whole real line
(namely, the Fourier transform of $\Psi$ seen as a function on
$\Rr$). We have then
\begin{equation} 
\label{eqn3.5}
\what\Psi'(h)= 2\pi i\int_{[0,1]}\alpha\Psi(\alpha) e(h\alpha)\,
d\alpha\ll \Bigl(\frac{N}{m}\Bigr)^2,
\end{equation}
for all $h$ and moreover, for $h\not=0$, we have (after integrating by parts $A\ge 1$ times as in 
\eqref{fouriercoeffbound}) 
\begin{align}\label{fouriercoeffboundderiv}\what\Psi'(h)&\ll_A  \frac{N}{m|h|} (U|h|)^{-A+1}. 
\end{align}

\par We now begin the estimation of the partial sums of $\what\vphi$. We
have
\begin{align} 
\label{eqn3.7}
\sum_{1\leq n\leq N}\what \varphi(n)& =\sum_{0\leq
    n<m}\what \varphi(n)\Psi\Bigl(\frac{n}{m}\Bigr)+
  O\bigl(\|\what \varphi\|_\infty \,m\,U\bigr) \nonumber \\
  &= m^{1/2}\sum_{h\in\Zz}\what\Psi(h) \varphi(h)+O\bigl(c\,m\,U\bigr),
\end{align}
where the implied constant is absolute, and where the second step (a version of the Plancherel formula) follows upon using the Fourier expansion \eqref{fourier}. 

\par 
By~(\ref{eq-psi0}), the term $h=0$ in \eqref{eqn3.7} equals
\begin{equation*} 
\label{eqn3.8} 
\sqrt{m} \varphi(0)\Big( \frac{N-1}{m} +O(U)\Big) \ll
\|\varphi\|_\infty\frac{N}{\sqrt{m}}\ll\frac{cN}{\sqrt{m}}.
\end{equation*} 
Next consider the contribution of the positive values of $h$.  By partial summation, these terms contribute 
\begin{align*}
-\sqrt{m} \int_{1^-}^{\infty} &S(\vphi, t)  {\what \Psi}^{\prime}(t) dt \\
&\ll \sqrt{m} \Big( \frac{N^2}{m^2} \int_1^{m/N} |S(\vphi,t)| dt  
+ \frac{N}{m} \int_{m/N}^{1/U} \frac{|S(\vphi,t)|}{t} dt + \frac{N}{mU^3}\int_{1/U}^{\infty} \frac{|S(\vphi,t)|}{t^4} dt \Big),  
\end{align*} 
upon using the estimate \eqref{eqn3.5} in the range $t\le m/N$, and the estimate \eqref{fouriercoeffboundderiv} with $A=1$ in the range $m/N\le t\le 1/U$ and with $A=4$ when $t>1/U$.   An analogous estimate holds for the contribution of negative $h$ to \eqref{eqn3.7}, and gathering these estimates together we obtain 
$$ 
|S(\what \vphi, N)| \ll cmU + c\frac{N}{\sqrt{m}} + \frac{N}{\sqrt{m}} \int_{1}^{\infty} 
\Big( | S(\vphi, t) | + |S(\vphi,-t)| \Big) \min \Big( \frac{N}{m}, \frac{1}{t}, \frac{1}{t^4 U^3}\Big) dt. 
$$ 
To complete the proof the proposition, it remains to bound the portion of the integral with $t>m/2$.  Write $t>m/2$ as $t=u+km$ where $k\ge 1$ is an integer, and $|u|\le m/2$.  Then by dividing the intervals $[1,t]$ and $[-t,-1]$ into complete intervals of length $m$ with intervals of length $|u|$ left over we see that 
$$ 
|S(\vphi,t)| +|S(\vphi,-t)| \ll |S(\vphi,u)|+ |S(\vphi,-u)| + \frac{ct}{\sqrt{m}}, 
$$ 
since the sum of $\vphi$ over a complete interval is $\le c\sqrt{m}$ in size.  It follows that the terms $t>m/2$ in the integral contribute 
\begin{align*}
&\ll \frac{N}{\sqrt{m}} \Big( \int_{m/2}^{\infty} \frac{ct}{\sqrt{m}} \frac{dt}{t^4U^3} + \int_1^{m/2} (|S(\vphi,u)|+|S(\vphi,-u)| ) \sum_{k=1}^{\infty}\frac{1}{(km)^4 U^3} du \Big) \\
&\ll \frac{cN}{m^{3} U^3} + \frac{N}{\sqrt{m}} \int_1^{m/2}   (|S(\vphi,u)|+|S(\vphi,-u)| ) \frac{du}{m^4U^3}. 
\end{align*}
The proposition follows.
\end{proof} 
 
 We are now ready for the proof of Theorem \ref{fouriertransfer}.  
 \begin{proof}[Proof of Theorem \ref{fouriertransfer}]  We first demonstrate \eqref{eqn3.12}, by an application of Proposition \ref{prop3.1}.  
  To estimate the integral in Proposition \ref{prop3.1}, we bound $|S(\phi,t)| + |S(\phi,-t)|$ 
  for $1\le t\le m/2$ by 
 $$ 
 \max_{t\le m/N} \Big( |S(\vphi,t)| + |S(\vphi,-t)|\Big) \le c \frac mN \Delta\Big(\vphi,\frac mN\Big), 
 $$  
and 
$$ 
\max_{m/2\ge t\ge m/N} \frac{1}{t} \Big( |S(\vphi,t)| + |S(\vphi,-t)|\Big) \le c  \Delta\Big(\vphi,\frac mN\Big).
$$ 
Thus, it follows that 
$$ 
|S(\what \vphi,N)| + |S(\what \vphi, -N)| \ll cmU+ c\frac{N}{\sqrt{m}} + \frac{cN}{m^3U^3}+c  \frac{N}{\sqrt{m}} \frac{1}{U} \Delta\Big(\vphi, \frac{m}{N}\Big). 
$$ 
Now choose $U= N^{\frac 12} m^{-\frac 34} \Delta(\vphi,m/N)^{\frac 12}$; since $\Delta(\vphi,m/N) \ge 1/\sqrt{m}$, it follows that 
$U\ge N^{\frac 12}m^{-1} \ge  1/m$,  and we may also assume that $U\le N/(2m)$ else the estimate \eqref{eqn3.12} holds trivially.   Thus our bound above applies, and it gives (noting that $\Delta(\vphi,m/N)/(\sqrt{m}U) \ge 1/(mU) \ge 1/(mU)^3$) 
$$ 
|S(\what \vphi,N)| + |S(\what \vphi, -N)| \ll c\sqrt{N} m^{\frac 14} \Delta\Big(\vphi, \frac mN\Big)^{\frac 12} + c\frac{N}{\sqrt{m}} \ll c\sqrt{N}m^{\frac 14} \Delta\Big( \vphi, \frac{m}{N}\Big)^{\frac 12}. 
$$ 

Thus we have established \eqref{eqn3.12}, and with it in hand, it is a simple matter to verify the second assertion of the theorem.  
If $t \le N$, then by \eqref{eqn3.12} we obtain 
$$ 
\frac{1}{cN} \Big(|S(\what\vphi,t)| + |S(\what \vphi, -t)|\Big) \ll \frac{\sqrt{t} m^{\frac 14}}{N} \Delta\Big(\vphi, \frac{m}{t}\Big)^{\frac 12} \ll \frac{m^{\frac 14}}{\sqrt{N}} \Delta\Big(\vphi, \frac mN\Big)^{\frac 12},
$$ 
since $\Delta(\vphi,m/t)$ is a non-decreasing function of $t$.  Finally if $N\le t\le m/2$ then 
$$ 
\frac{1}{ct} \Big(|S(\what\vphi,t)| + |S(\what \vphi, -t)|\Big) \ll \frac{m^{\frac 14}}{\sqrt{t}} \Delta\Big(\vphi, \frac mt\Big)^{\frac 12} = m^{-\frac 14} \Big( \frac mt \Delta \Big(\vphi,\frac mt\Big)\Big)^{\frac 12} \le m^{-\frac 14} 
\Big( \frac mN \Delta\Big(\vphi,\frac mN\Big)\Big)^{\frac 12}, 
$$
since $(m/t) \Delta(\vphi, m/t)$ is a non-increasing function of $t$.  Combining these estimates, and noting that $1/\sqrt{m}$ is smaller than $(m^{\frac 14}/\sqrt{N}) \Delta(\vphi,m/N)^{\frac 12}$, we obtain the theorem.
\end{proof} 

\subsection{Applications} 
\label{sec-below-pv1}

\par
We will prove quantitative versions of Corollary~\ref{cor-below-pv},
in the sense of specifying the value of the quantity $\delta>0$ that
appears there. 
The argument splits naturally in two cases, depending on whether the
character $\chi$ in~(\ref{eq-function}) is trivial or not. We begin
with the former case, where we can in fact work with an arbitrary
squarefree modulus $m\geq 1$.

\begin{theorem}\label{fourierpoly}
  Let $m\geq 1$ be a squarefree integer. Let $P\in(\Zz/m\Zz)[X]$ be a
  polynomial of degree $d\geq 3$ with invertible leading
  coefficient. Let $\varphi\,:\,\Zz/m\Zz\lra \Cc$ be defined by
$$
\vphi(n)=e_m(P(n))
$$
and let
$$
\what \vphi(n)=\frac{1}{m^{1/2}}\sum_{1\leq h\leq m}e_m({P(h)+nh})
$$
be its Fourier transform.  
\par
For any $\eta< 1/(2^d-2)$ there exist $\delta>0$ depending only on
$\eta$ such that if $N\geq m^{1/2-\eta}$, we have
$$
\sum_{1\leq n\leq N}\what \vphi(n)\ll N^{1-\delta},
$$
where the implied constant depends only on $\eta$ and $d$.
\end{theorem}


\begin{proof}
  We begin by noting that a combination of the Weil bound for
  exponential sums with additive characters and of the Chinese
  Remainder Theorem shows that
$$
|\widehat{\varphi}(t)|\leq (d-1)^{\omega(m)}\ll m^{\eps}
$$
for any $\eps>0$ and any $t\in\Zz/m\Zz$, where the implied constant
depends only on $\eps$ and $d$ (we use here the fact that $P$ is of
degree $d$ modulo any prime divisor of $m$, since we assume that the
leading coefficient is invertible modulo $m$). In particular, we get
\begin{equation}\label{eq-fourier-sqf}
  c=\max(\|\varphi\|_{\infty},\|\widehat{\varphi}\|_{\infty})\ll
  m^{\eps}.
\end{equation}
\par
Let $\kappa=1/2^{d-1}$.  The key ingredient is the Weyl bound
\begin{equation}\label{song}
  \Bigl\vert \sum_{1\leq h \leq H } e_m({P(h)})
  \Bigr\vert \ll H^{1+\varepsilon}
  \Bigl(\frac{1}{H}+\frac{m}{H^d}\Bigr)^{\kappa}
\end{equation}
valid for an arbitrary $\eps>0$ and $1\leq H\leq m$, with an implied
constant that depends only on $d$ and $\eps$ (see~\cite[Lemma
20.3]{IK} or~\cite[Lemma 2.4]{Va}, and for recent bounds that are much
stronger for large $d$, see~\cite[Theorem 1.5]{Wooley}).  
This implies that for $1\leq t\leq m/2$, we have
$$
|S(\varphi,\pm t)|\ll \min\Bigl(|t|, |t|^{1+\varepsilon}
\Bigl(\frac{1}{|t|}+\frac{m}{|t|^d}\Bigr)^{\kappa} \Bigr)
$$
for any $\eps>0$, where the implied constant depends on $d$ and $\eps$
only. By~(\ref{defDelta}), this leads to
$$ 
\Delta(\vphi, H) \ll \frac{1}{\sqrt{m}}+\frac{m^{\eps}}{c} \Big(
\frac{1}{H^{\kappa}} + \frac{m^{\kappa}}{H^{d\kappa}}\Big)
$$
for any $\eps>0$, where the implied constant depends on $d$ and $\eps$
only.
\par
Appealing to \eqref{eqn3.12} from Theorem \ref{fouriertransfer}, we
conclude that
$$ 
\Big| \sum_{n\le N} {\what \vphi}(n) \Big| \ll cN^{1/2}+c^{1/2}
N^{\frac 12} m^{\frac 14+\epsilon} \Big(\frac{N^{\kappa}}{m^{\kappa}}
+ \frac{N^{d\kappa}}{m^{(d-1)\kappa}} \Big)^{\frac 12},
$$ 
which, with a small calculation using~(\ref{eq-fourier-sqf}), yields
the theorem.
\end{proof}

When the character $\chi$ in~(\ref{eq-function}) is non-trivial, we
will need to assume that the modulus is prime. 

\begin{theorem}\label{bingo}
  Let $p$ be a prime number. Let $P\in\Fp[X]$ be a polynomial of
  degree $d\geq 0$, let $\chi$ be a non-trivial multiplicative
  character modulo $p$ and let $m\in\Fp$. Define
  $\varphi\,:\, \Fp\lra \Cc$ by
$$
\vphi(n)=\chi(n+m)e_p(P(n))
$$
and let
$$
\what \vphi(n)=\frac{1}{p^{1/2}}\sum_{1\leq h\leq
  p}\chi(h+m)\,e_p(P(h)+nh)
$$
be its Fourier transform.
\par
For any $\eta<\frac{1}{8(d^2+d+1)}$ 
there exists $\delta>0$ such that for all $N$ with $N\geq
p^{1/2-\eta}$, we have
$$
\sum_{1\leq n\leq N}\what \vphi(n)\ll N^{1-\delta},
$$
where the implied constant depends only on $\eta$ and $d$.
\end{theorem}

\begin{proof}
The method is similar to the previous case. However, we now use
instead of the Weyl bound a recent result of Heath-Brown and Pierce,
namely
$$
\Bigl\vert \sum_{1\leq n \leq N } \chi(n+m)e_p( {P(n)}) \Bigr\vert
\ll_{d,r} (\log p)^2 \min\Bigl(\,p^{1/2}, p^{\frac{r+1+D}{4r^2}}
N^{1-\frac{1}r}\,\Bigl)
$$
where $D=d(d+1)/2$ (see~\cite[Th. 1.2]{HBP}, noting that the bound is
stated there only for $N\leq q^{1/2+1/4r}$, but that it becomes weaker
than the P{\' o}lya-Vinogradov bound when $N\geq p^{1/2+1/4r}$).  Thus 
$$ 
\Delta(\vphi, N) \ll_{d,r} (\log p)^2p^{\frac {r+1+D}{4r^2}}{N^{-\frac 1r}}. 
$$ 
\par
Applying \eqref{eqn3.12} of Theorem \ref{fouriertransfer}, we obtain
(for $N>p^{\frac 14}$)
$$ 
\Big| \sum_{n\le N} \what \vphi(n) \Big| \ll p^{\frac 14 +\frac{r+1+D}{8r^2}-\frac{1}{2r}}N^{\frac 12 +\frac{1}{2r}} (\log p) .
$$ 
Choosing $r=2(D+1)$, the theorem follows.  
\end{proof} 

\begin{remark}
  The bound of Heath-Brown and Pierce is the latest of a series of
  works by Enflo, Heath-Brown and Chang,
  see~\cites{Enflo,HBreview,chang}, any one of which would lead the
  qualitative form in Corollary~\ref{cor-below-pv}.
\end{remark}

We now consider the special case of the cubic Birch sums to prove
Corollary~\ref{cor-distrib-birch}. 

\begin{proof}[Proof of Corollary~\ref{cor-distrib-birch}]
  Let $p$ be prime and let $\varphi$ be defined on $\Fp$ by
$$
\vphi(h)=e_p( {h^3} ).
$$ 
As before, we denote by
$$
\rmB_3(n)=\frac{1}{p^{1/2}}\sum_{1\leq h\leq p}e_p( h^3+nh )
$$
its Fourier transform. Note that for any fixed $m\in\Zz$, the shifted
Birch sums $n\mapsto \rmB_3(n+m)$ is also a Fourier transform of a
polynomial, namely of
$$
h\mapsto e_p(h^3+hm).
$$
\par
Thus the estimate~(\ref{eq-mom1-birch}) for the first moment is a
special case of Theorem~\ref{fourierpoly}: for any $\eta<1/6$, there
exists $\delta>0$, depending only on $\eta$, such that
$$
\sum_{m\leq n\leq m+N}\rmB_3(n)\ll N^{1-\delta}
$$
for all $m\in\Zz$, provided $N > p^{1/2 -\eta}$.
\par
We now consider the second moment. We assume $p\geq 5$ and we define
$$
\psi(n)=|\rmB_3(n)|^2-1,
$$
and compute its Fourier transform. For any $h\in\Fpt$, we have
\begin{align*}
  \what \psi(h)&=\frac{1}{p^{1/2}} \sum_{n\in\Fp}
                 (|\rmB_3(n)|^2-1)e_p(nh) \\
               & =\frac{1}{p^{3/2}}\sum_{u,v,n\in\Fp}
                 e_p\bigl(u^3-v^3+n(u-v+h)\bigr)
                 =\frac{1}{p^{1/2}}\sum_{u\in\Fp}e_p(u^3-(u+h)^3 ).
\end{align*}
This is a quadratic complete sum, hence it can be evaluated
exactly. Precisely, we obtain by completing the square the formula
$$
\what \psi(h)=\eps_p \Bigl(\frac{h}{p}\Bigr)e_p (-h^3/4)
$$
for $h\in\Fpt$, where $\eps_p$ is a complex number of modulus one
independent of $h$. In fact, the same formula holds for $h=0$, as one
checks immediately by a similar computation. 
\par
By the discrete Fourier inversion formula, we deduce that $\psi$ is
the Fourier transform of the function $n\mapsto \eps_p
\bigl(\frac{n}{p}\bigr)e_p (n^3/4)$. Hence Theorem~\ref{bingo}
implies  (after using an
additive shift by $m$ as above to get the estimate for any interval)
$$
\sum_{m\leq n\leq m+N}(|\rmB_3(n)|^2-1)\ll N^{-\delta}
$$
for any $\eta<1/104$ and some $\delta>0$ depending only on $\eta$,
which is more precise than~\eqref{eq-mom2-birch}.
\par
We now prove the final part of Corollary~\ref{cor-distrib-birch}
concerning the distribution of values of Birch sums.
Let $\eta<1/104$ be fixed and let $I$ be an interval in $\Zz$ with
$|I|\geq p^{1/2-\eta}$.  From our work above, we know that for some $\delta>0$ (depending only on $\eta$) we have 
$$ 
(1+O(p^{-\delta}))|I| = \sum_{n\in I} |\rmB_3(n)|^2 \le t^2 \sum_{\substack{n\in I \\  \rmB_3(n)\in [0,t]}} 1 + 4\sum_{\substack{n \in I \\ \rmB_3(n)\notin [0,t]}} 1, 
$$ 
where we have used the Weil bound $|\rmB_3(n)| \le 2$.  With a little rearranging, this yields 
\begin{equation} 
\label{eqn3.16} 
\sum_{\substack {n\in I \\ \rmB_3(n)\in [0,t]}} 1 \le \Big(\frac{3}{4-t^2}+o(1)\Big) |I|. 
\end{equation} 
On the other hand, since $\sum_{n\in I} \rmB_3(n) = o(|I|)$ we have  
$$ 
\sum_{\substack {n\in I \\ \rmB_3(n) <0} } \rmB_3(n)^2  \le 2 \sum_{\substack{ n\in I \\ \rmB_3(n) <0}} |\rmB_3(n)| 
= 2 \sum_{\substack{ n\in I \\ \rmB_3(n)\ge 0}} \rmB_3(n) + o(|I|) \le 2t \sum_{\substack{ n \in I \\ \rmB_3(n) \in [0,t]}} 1 + 
4 \sum_{\substack{n \in I \\ \rmB_3(n)\in (t,2]}} 1 +o(|I|), 
$$  
and therefore 
$$ 
(1+o(1))|I| = \sum_{ n\in I} |\rmB_3(n)|^2 \le t^2 \sum_{\substack{n\in I \\ \rmB_3(n)\in [0,t]}} 1 + 4 \sum_{\substack{n\in I\\ \rmB_3(n) \in (t,2]}} 1 + 2t 
\sum_{\substack{ n \in I \\ \rmB_3(n) \in [0,t]}} 1 + 
4 \sum_{\substack{n \in I \\ \rmB_3(n)\in (t,2]}} 1 +o(|I|).
$$ 
Combining this with \eqref{eqn3.16}, we deduce that (for $t<1/2$) 
$$ 
\sum_{\substack{n\in I \\ \rmB_3(n) >t}} 1 \ge \frac{1}{8} \Big( 1- \frac{3t}{2-t} + o(1)\Big)  |I| = \Big( \frac{1-2t}{4(2-t)} + o(1)\Big) |I|, 
$$ 
as desired.  The bound on the number of $\rmB_3(n) <-t$ is obtained similarly.  The last bound on the frequency of $n$ with $|\rmB_3(n)| \ge t$ follows upon noting that 
$$ 
|I| \sim \sum_{n\in I} |\rmB_3(n)|^2 \le t^2 \sum_{\substack{ n \in I\\ |\rmB_3(n)| \le t}} 1 + 4 \sum_{\substack{ n \in I\\ |\rmB_3(n)| > t}} 1 \sim t^2 |I| + (4-t^2)  \sum_{\substack{ n \in I\\ |\rmB_3(n)| > t}} 1.
$$ 
\end{proof}


Finally, we prove Corollary \ref{BouGarFou}. 

\begin{proof}[Proof of Corollary~\ref{BouGarFou}]
  Let $p$ be prime and $a\in\Fpt$. Let $k=1$ or $k=2$, and define
$$
\varphi(h)=e_p(a{ h}^{-k})
$$
for $h\in\Fpt$ and $\varphi(0)=0$. Then
$$\what{\varphi}(n)=
\frac{1}{p^{1/2}}\sum_{h\in\Fpt}e_p({a h^{-k}+nh}).
$$

In~\cites{BoGa,BouPell}, Bourgain and Garaev have obtained non-trivial
estimates for sums of $\varphi$.
Precisely, they proved that
there exist absolute constants $\delta>0$ and $0<\eta<1$ such that
\begin{equation}\label{Bour}
  \max_{a\in\Fpt}\, \Bigl\vert\, 
  \sum_{1\leq h\leq H}e_p\bigl(a{h^{-k}}\bigr) \Bigr\vert \ll 
  H{(\log p)^{-\delta}}
\end{equation}
provided $H\geq \exp((\log p)^{\eta})$.  Thus, if
$H\ge \exp((\log p)^{\eta})$, then
$\Delta(\vphi,H)\ll (\log p)^{-\delta}$.
\par  
Applying \eqref{eqn3.12} of Theorem \ref{fouriertransfer}, we obtain
for
$p^{\frac 12} (\log p)^{-\frac{\delta}{2}} \le N \le p^{\frac 12}(\log
p)^2$,
$$ 
S(\varphi,N)\ll \sqrt{N} p^{\frac 14} (\log p)^{-\frac {\delta}{2}}
\ll N (\log p)^{-\frac{\delta}{4}}.
$$ 
Since this bound holds also in the range
$N\ge p^{\frac 12} (\log p)^2$ by P{\' o}lya-Vinogradov, our proof is
complete.
\end{proof}




\begin{thebibliography}{99}

\bib{birch}{article}{
 author={Birch, B.J.},
 title={How the number of points on an elliptic curve over a fixed
   prime prime varies},
 journal={J. London Math. Soc.},
 volume={43},
 year={1968},
 pages={57--60},
}

\bib{BFKMM}{article}{
 author={Blomer,V.},
 author={Fouvry, \' E.},
 author={Kowalski, E.},
 author={Michel, Ph.},
 author={Mili\'cevi\'c,  D.},
 title={ On moments of twisted $L$--functions},
 journal={American J. Math., to appear},
 note={\url{arXiv:1411.4467}},
}

\bib{BouPell}{article}{
   author={Bourgain, J.},
   title={A Remark on Solutions of the Pell Equation},
   journal={Int. Math. Res. Not. IMRN},
   date={2015},
   number={10},
   pages={2841--2855},
 
}
\bib{BoGa}{article}{
  author={Bourgain, J.},
  author={Garaev, M. Z.},
  title={Sumsets of reciprocals in prime fields and multilinear Kloosterman
    sums},
  language={Russian, with Russian summary},
  journal={Izv. Ross. Akad. Nauk Ser. Mat.},
  volume={78},
  date={2014},
  number={4},
  pages={19--72},
  issn={0373-2436},
  translation={
    journal={Izv. Math.},
    volume={78},
    date={2014},
    number={4},
    pages={656--707},
    issn={1064-5632},
  },
}

\bib{Bu1}{article}{
  author={Burgess, D. A.},
  title={On character sums and $L$-series II},
  journal={Proc. London Math. Soc. (3)},
  volume={13},
  date={1963},
  pages={524--536},
}


\bib{chang}{article}{
  author={Chang, M.--C.},
  title={An estimate of incomplete mixed character sums},
  conference={
    title={An irregular mind},
  },
  book={
    series={Bolyai Soc. Math. Stud.},
    volume={21},
    publisher={J\'anos Bolyai Math. Soc., Budapest},
  },
  date={2010},
  pages={243--250},
}
   
\bib{Enflo}{article}{
  author={Enflo, P.},
  title={Some problems in the interface between number theory, harmonic
    analysis and geometry of Euclidean space},
  note={First International Conference in Abstract Algebra (Kruger Park,
    1993)},
  journal={Quaestiones Math.},
  volume={18},
  date={1995},
  number={1-3},
  pages={309--323},
}



\bib{FKM1}{article}{
  author={Fouvry, {\'E}.},
  author={Kowalski, E.},
  author={Michel, Ph.},
  title={Algebraic twists of modular forms and Hecke orbits},
  journal={Geom. Funct. Anal.},
  volume={ 25 },
  note={{\url{arXiv:1207.0617  }}},
  date={2015},
  number={2},
  pages={580--657},
}

\bib{FKM2}{article}{
  author={Fouvry, {\'E}.},
  author={Kowalski, E.},
  author={Michel, Ph.},
  title={Algebraic trace functions over the primes},
  journal={Duke Math. J.},
  volume={163},
  note={{\url{arXiv:1211.6043}}},
  date={2014},
  number={9},
  pages={1183--1736},
}

\bib{FKM3}{article}{
  author={Fouvry, {\'E}.},
  author={Kowalski, E.},
  author={Michel, Ph.},
  title={On the exponent of distribution of the ternary divisor function},
  journal={Mathematika},
  volume={61},
  date={2015},
  number={1},
  pages={121--144},
}
  
\bib{counting-sheaves}{article}{
  author={Fouvry, {\'E}.},
  author={Kowalski, E.},
  author={Michel, Ph.},  
  title={Counting sheaves using spherical codes},
  journal={ Math. Res. Lett.},
  volume={20},
  date={2013},
  number={2},
  pages={305--323},
}  

\bib{gowers-norms}{article} {
  author={Fouvry, {\'E}.},
  author={Kowalski, E.},
  author={Michel, Ph.}, 
  title={An inverse theorem for Gowers norms of trace functions over
    $\mathbf{F}_p$}, 
  journal={Math. Proc. Cambridge Philos. Soc.},
  volume={155},
  date= {2013},
  number={2},
  pages={277--295},
}

\bib{pisa}{article} {
  author={Fouvry, {\'E}.},
  author={Kowalski, E.},
  author={Michel, Ph.}, 
  title={Trace functions over finite fields and their applications}, 
  journal={Colloquium de Giorgi, Publications of the Scuola Normale Superiore de Pisa},
  volume={5,},
  date= {2015},
  pages={7--35},
}

\bib{FroSo}{article} {
    AUTHOR = {Frolenkov, D. A.},
    author={Soundararajan, K.},
     TITLE = {A generalization of the {P}\'olya-{V}inogradov inequality},
   JOURNAL = {Ramanujan J.},
    VOLUME = {31},
      YEAR = {2013},
    NUMBER = {3},
     PAGES = {271--279},
      ISSN = {1382-4090},
}

 
\bib{HBreview}{article}{
  author={Heath-Brown, D. R.},
  title={MR review of "Some problems in the interface between number theory, harmonic
    analysis and geometry of Euclidean space" by P. Enflo},
  review={\MR{1340486 (96h:11079)}},
}

\bib{HBP}{article}{
  author={Heath-Brown, D. R.},
  author={Pierce, L.},
  title={Burgess bounds for short mixed character sums},
  journal= {J. London Math.Soc. (2)},
  volume={91},
  date={2015},
  number={2},
  pages={693--708},
}

\bib {IK} {book}{
  author={Iwaniec, H.},
  author={Kowalski, E.},  
  title={Analytic Number Theory}, 
  series={American Mathematical Society Colloquium Publications},
  volume={53},
  publisher={American Mathematical Society, Providence, RI},
  address={Providence, RI},
  date={2004},
}


\bib{katz-gskm}{book}{
  author={Katz, N. M.},
  title={Gauss sums, Kloosterman sums and monodromy groups},
  series={Annals of Mathematics  Studies},
  publisher= {Princeton University Press},
  address={Princeton, N.J.},
  volume={116},
  date={1988},
}

 
\bib{katz-mellin} {book}{
  author={Katz, N. M.},
  title={Convolution and  equidistribution: Sato-Tate theorems of finite-field Mellin transforms},
  series={Annals of Mathematics  Studies},
  publisher= {Princeton University Press},
  address={Princeton, N.J.},
  volume={180},
  date={2012},
}
  
\bib{korobov}{book}{ 
  author={Korobov, N. M.},
  title={Exponential sums and their applications},
  series={Mathematics and its Applications (Soviet Series)},
  volume={80},
  publisher={ Kluwer Academic Publishers Group},
  address= {Dordrecht},
  date={1992},
}



\bib{k-sawin}{article}{
  author={Kowalski, E.},
  author={Sawin, W.},
  title={Kloosterman paths and the shape of exponential sums},
  journal={Compositio Math., to appear},
  year={2016},
  note={\url{doi:10.1112/S0010437X16007351}},
}

\bib{livne}{article}{
  author={Livn{\'e}, R.},
  title={The average distribution of cubic exponential sums},
  journal={J. reine angew. Math.},
  volume={375--376},
  year={1987},
  pages={362--379},
}


\bib{Va}{book}{  
  author={Vaughan, R. C.},
  title={The Hardy--Littlewood method. Second Edition},
  series={Cambridge Tracts in Mathematics},
  volume={125},
  publisher={ Cambridge University Press},
  address={Cambridge},
  date={1997},
  }
  
  \bib{Wooley}{article}{
    AUTHOR = {Wooley, T. D.},
     TITLE = {Vinogradov's mean value theorem via efficient congruencing},
   JOURNAL = {Ann. of Math. (2)},
    VOLUME = {175},
      YEAR = {2012},
    NUMBER = {3},
     PAGES = {1575--1627},
      ISSN = {0003-486X},
       DOI = {10.4007/annals.2012.175.3.12},
       URL = {http://dx.doi.org/10.4007/annals.2012.175.3.12},
}
  
  

\end{thebibliography}
\end{document}